\documentclass[12pt,a4paper]{article}

\usepackage[utf8]{inputenc}
\usepackage[english]{babel}
\usepackage{graphicx}

\usepackage{geometry}
\usepackage{amsmath, amsthm}
\usepackage{amssymb,amsfonts}

\newtheorem{definition}{Definition}
\newtheorem{theorem}[definition]{Theorem}
\newtheorem{lemma}[definition]{Lemma}
\newtheorem{corollary}[definition]{Corollary}

\renewcommand{\vert}{\mathrm{vert}}
\renewcommand{\dim}{\mathrm{dim}}
\newcommand{\carr}{\mathrm{carr}}

\renewcommand{\epsilon}{\ensuremath\varepsilon}

\newcommand{\outmap}{\phi}
\newcommand{\outmapp}{\psi}

\newcommand{\Cube}{\mathcal{C}}
\newcommand{\Face}{\mathcal{F}}

\newcommand{\odd}{\mathrm{odd}}
\newcommand{\uso}{\mathrm{uso}}
\newcommand{\puso}{\mathrm{puso}}
\newcommand{\border}{\mathrm{border}}
\newcommand{\KM}{\mathrm{KM}}
\newcommand{\coNP}{\textbf{coNP}}

\newcommand{\orient}{\mathcal{O}}

\title{Pseudo Unique Sink Orientations}
\author{Vitor Bosshard \and Bernd G\"artner}
\date{\today}

\begin{document}

\maketitle

\begin{abstract}
  A unique sink orientation (USO) is an orientation of the
  $n$-dimensional cube graph ($n$-cube) such that every face (subcube)
  has a unique sink. The number of unique sink orientations is
  $n^{\Theta(2^n)}$~\cite{matousek2006number}. If a cube orientation
  is not a USO, it contains a \emph{pseudo unique sink orientation}
  (PUSO): an orientation of some subcube such that every proper face
  of it has a unique sink, but the subcube itself hasn't. In this
  paper, we characterize and count PUSOs of the $n$-cube. We show that
  PUSOs have a much more rigid structure than USOs and that their
  number is between $2^{\Omega(2^{n-\log n})}$ and $2^{O(2^n)}$ which
  is negligible compared to the number of USOs. As tools, we introduce
  and characterize two new classes of USOs: \emph{border USOs} (USOs
  that appear as facets of PUSOs), and \emph{odd USOs} which are dual
  to border USOs but easier to understand.
\end{abstract}

\section{Introduction}

\paragraph{Unique sink orientations.}
Since more than 15 years, unique sink orientations (USOs) have been
studied as particularly rich and appealing combinatorial abstractions
of linear programming (LP)~\cite{gartner2006linear} and other
related problems~\cite{fischer2004}. Originally introduced by
Stickney and Watson in the context of the P-matrix linear
complementarity problem (PLCP) in 1978~\cite{StiWat}, USOs have been
revived by Szab\'o and Welzl in 2001, with a more theoretical
perspective on their structural and algorithmic
properties~\cite{szabo2001unique}.

The major motivation behind the study of USOs is the open question
whether efficient combinatorial algorithms exist to solve PLCP and
LP. Such an algorithm is running on a RAM and has runtime bounded by a
polynomial in the \emph{number} of input values (which are considered
to be real numbers). In case of LP, the runtime should be polynomial
in the number of variables and the number of constraints. For LP, the
above open question might be less relevant, since polynomial-time
algorithms exist in the Turing machine model since the breakthrough
result by Khachiyan in 1980~\cite{Kha}. For PLCP, however, no such
algorithm is known, so the computational complexity of PLCP remains
open.

Many algorithms used in practice for PCLP and LP are combinatorial and
in fact \emph{simplex-type} (or \emph{Bard-type}, in the LCP
literature). This means that they follow a locally improving path of
candidate solutions until they either cycle (precautions need to be
taken against this), or they get stuck---which in case of PLCP and LP
fortunately means that the problem has been solved. The less fortunate
facts are that for most known algorithms, the length of the path is
exponential in the worst case, and that for no algorithm, a
polynomial bound on the path length is known. 

USOs allow us to study simplex-type algorithms in a completely
abstract setting where cube vertices correspond to candidate
solutions, and outgoing edges lead to locally better
candidates. Arriving at the unique sink means that the problem has
been solved. The requirement that all faces have unique sinks is
coming from the applications, but is also critical in the abstract
setting itself: without it, there would be no hope for nontrivial
algorithmic results~\cite{Aldous}.

On the one hand, this kind of abstraction makes a hard problem even
harder; on the other hand, it sometimes allows us to see what is
really going on, after getting rid of the numerical values that hide
the actual problem structure. In the latter respect, USOs have been
very successful. 

For example, in a USO we are not confined to following a path, we can
also ``jump around''. The fastest known deterministic algorithm for
finding the sink in a USO does exactly this~\cite{szabo2001unique} and
implies the fastest known deterministic combinatorial algorithm for LP
if the number of constraints is twice the number of
variables~\cite{gartner2006linear}. In a well-defined sense, this is
the hardest case. Also, \textsc{RandomFacet}, the currently best
randomized combinatorial simplex algorithm for LP~\cite{Kal,MSW}
actually works on acyclic USOs (AUSOs) with the same (subexponential) runtime
and a purely combinatorial analysis~\cite{gartner2002}.

The USO abstraction also helps in proving lower bounds for the
performance of algorithms. The known (subexponential) lower bounds for
\textsc{RandomFacet} and \textsc{RandomEdge}---the most natural
randomized simplex algorithm---have first been proved on 
AUSOs~\cite{Mat,MATOUSEK2006262} and only later on actual linear
programs~\cite{Friedmann}. It is unknown which of the two algorithms
is better on actual LPs, but on AUSOs, \textsc{RandomEdge} is
strictly slower in the worst case~\cite{HZ}.

Finally, USOs are intriguing objects from a purely mathematical point
of view, and this is the view that we are mostly adopting in in this paper.

\paragraph{Pseudo unique sink orientations.} If a cube orientation has
a unique sink in every face except the cube itself, we call it a pseudo
unique sink orientation (PUSO). Every cube orientation that is not a
USO contains some PUSO. The study of PUSOs originates from the
master's thesis of the first author~\cite{Bosshard} where the PUSO
concept was used to obtain improved USO recognition algorithms;
see Section~\ref{sec:recognition} below.

One might think that PUSOs have more variety than USOs: instead of
exactly one sink in the whole cube, we require any number of sinks not
equal to one. But this intuition is wrong: as we show, the number of
PUSOs is much smaller than the number of USOs of the same dimension;
in particular, only a negligible fraction of all USOs of one dimension
lower may appear as facets of PUSOs. These \emph{border USOs} and the
\emph{odd USOs}---their duals---have a quite interesting structure
that may be of independent interest. The discovery of these USO
classes and their basic properties, as well as the implied counting
results for them and for PUSOs, are the main contributions of the
paper.

\paragraph{Overview of the paper.} Section~\ref{sec:cubes} formally
introduces cubes and orientations, to fix the language. We will define
an orientation via its \emph{outmap}, a function that yields for
every vertex its outgoing edges. Section~\ref{sec:usos} defines USOs
and PUSOs and gives some examples in dimensions two and three to
illustrate the concepts. In Section~\ref{sec:outmaps}, we characterize
outmaps of PUSOs, by suitably adapting the characterization for USOs
due to Szab\'o and
Welzl~\cite{szabo2001unique}. Section~\ref{sec:recognition} uses the
PUSO characterization to describe a USO recognition algorithm that is
faster than the one resulting from the USO characterization of Szab\'o
and Welzl. Section~\ref{sec:border} characterizes the USOs that may
arise as facets of PUSOs. As these are on the border between USOs and
non-USOs, we call them \emph{border USOs.} Section~\ref{sec:odd}
introduces and characterizes the class of \emph{odd} USOs that are
dual to border USOs under inverting the outmap. Odd USOs are easier to
visualize and work with, since in any face of an odd USO we again have
an odd USO, a property that fails for border USOs. We also give a
procedure that allows us to construct many odd USOs from a canonical
one, the \emph{Klee-Minty cube}. Based on this,
Section~\ref{sec:counting} proves (almost matching) upper and lower
bounds for the number of odd USOs in dimension $n$. Bounds on the
number of PUSOs follow from the characterization of border USOs in
Section~\ref{sec:border}. In Section~\ref{sec:conclusion}, we mention
some open problems. 
 
\section{Cubes and Orientations}\label{sec:cubes}

Given finite sets $A\subseteq B$, the \emph{cube}
$\Cube=\Cube^{[A,B]}$ is the graph with vertex set
$\vert{\Cube} = [A,B]:=\{V: A\subseteq V \subseteq B\}$ and edges
between any two subsets $U,V$ for which $|U\oplus V|=1$, where
$U\oplus V=(U\setminus V)\cup (V\setminus U) = (U\cup V)\setminus
(U\cap V)$ is symmetric
difference. We sometimes need the following easy fact.
\begin{equation}\label{eq:symdiff}
(U\oplus V)\cap X = (U\cap X) \oplus (V\cap X).
\end{equation}

For a cube $\Cube=\Cube^{[A,B]}$, $\dim{\Cube}:=|B\setminus A|$ is its
\emph{dimension}, $\carr{\Cube} := B\setminus A$ its \emph{carrier}.
A \emph{face} of $\Cube$ is a subgraph of the form
$\Face=\Cube^{[I,J]}$, with $A\subseteq I\subseteq J\subseteq B$. If
$\dim{\Face}=k$, $\Face$ is a \emph{$k$-face} or \emph{$k$-cube}. A
\emph{facet} of an $n$-cube $\Cube$ is an $(n-1)$-face of $\Cube$.
Two vertices $U,V\in\vert{\Face}$ are called \emph{antipodal} in
$\Face$ if $V=\carr{\Face}\setminus U$.

If $A=\emptyset$, we abbreviate $\Cube^{[A,B]}$ as $\Cube^B$. The
\emph{standard $n$-cube} is $\Cube^{[n]}$ with
$[n]:=\{1,2,\ldots,n\}$.

An \emph{orientation} $\orient$ of a graph $G$ is a digraph that
contains for every edge $\{U,V\}$ of $G$ exactly one directed edge $(U,V)$
or $(V,U)$. An orientation of a cube $\Cube$ can be specified by its
\emph{outmap} $\outmap: \vert{\Cube} \rightarrow 2^{\carr{\Cube}}$
that returns for every vertex the \emph{outgoing coordinates}. On
every face $\Face$ of $\Cube$ (including $\Cube$ itself), the outmap
induces the orientation
\[\Face_{\outmap} := (\vert{\Face}, \{(V,V\oplus\{i\}): V\in \vert{\Face},
i\in \outmap(V)\cap \carr{\Face}\}).
\]
In order to actually get a proper orientation of $\Cube$, the outmap
must be \emph{consistent}, meaning that it satisfies
$i\in \outmap(V)\oplus\outmap(V\oplus\{i\})$ for all
$V\in\vert{\Cube}$ and $i\in\carr{\Cube}$.

Note that the outmap of $\Face_{\outmap}$ is not $\phi$
but $\outmap_{\Face}:\vert{\Face}\rightarrow 2^{\carr{\Face}}$ defined by 
\begin{equation}\label{eq:faceout}
\outmap_{\Face}(V)=\outmap(V)\cap\carr{\Face}.
\end{equation}

In general, when we talk about a cube orientation
$\orient=\Cube_{\outmap}$, the domain of $\outmap$ may be a supercube
of $\Cube$ in the given context. This avoids unnecessary indices that
we would get in defining $\orient=\Cube_{\outmap_{\Cube}}$ via its
``official'' outmap $\outmap_{\Cube}$. However, sometimes we want to
make sure that $\outmap$ is actually the outmap of $\orient$, and then
we explicitly say so.

Figure~\ref{fig:orientation} depicts an outmap and the corresponding 2-cube orientation. 

\begin{figure}[htb]
\begin{center}
\includegraphics[width=0.7\textwidth]{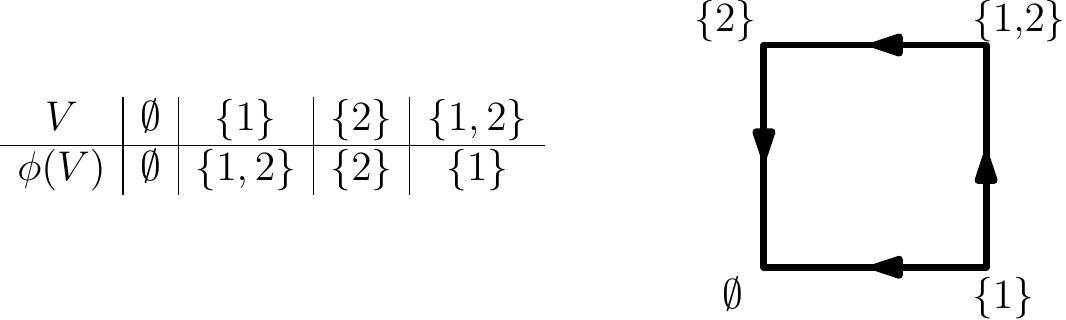}
\end{center}
\caption{An outmap $\outmap$ and the induced 2-cube orientation $\Cube_{\outmap}$}
\label{fig:orientation}
\end{figure}

\section{(Pseudo) Unique Sink Orientations}\label{sec:usos}

\begin{definition}[USO~\cite{szabo2001unique}] 
  A unique sink orientation (USO) of a cube $\Cube$ is an orientation
  $\Cube_{\outmap}$ such that every face $\Face_{\outmap}$ has a
  unique sink. Equivalently, every face $\Face_{\outmap}$ is a unique
  sink orientation.
\end{definition}

Figure~\ref{fig:2cubes} shows the four combinatorially different
(pairwise non-isomorphic) orientations of the 2-cube. The eye and the
bow are USOs.\footnote{The naming goes back to Szab\'o and
  Welzl~\cite{szabo2001unique}.} The twin peak is not since it has two
sinks in the whole cube (which is a face of itself). The cycle is not
a USO, either, since it has no sink in the whole cube. The unique sink
conditions for $0$- and $1$-faces (vertices and edges) are always
trivially satisfied.

\begin{figure}[htb]
\begin{center}
\includegraphics[width=\textwidth]{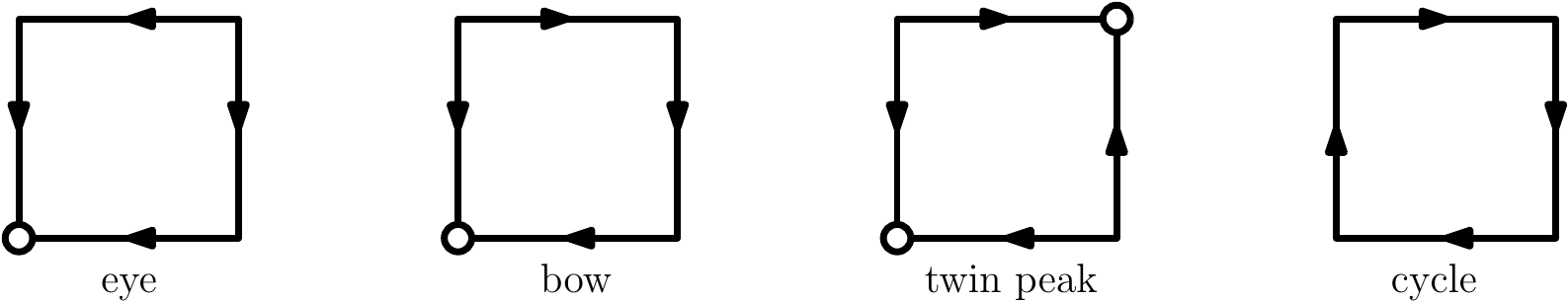}
\end{center}
\caption{The 4 combinatorially different orientations of the
  2-cube}
\label{fig:2cubes}
\end{figure}

If an orientation $\Cube_{\outmap}$ is not a USO, there is a smallest
face $\Face_{\outmap}$ that is not a USO. We call the orientation in
such a face a \emph{pseudo} unique sink orientation.

\begin{definition}[PUSO]
  A pseudo unique sink orientation (PUSO) of a cube $\Cube$ is an
  orientation $\Cube_{\outmap}$ that does not have a unique sink, but
  every proper face $\Face_{\outmap} \neq \Cube_{\outmap}$ has a
  unique sink.
\end{definition}

The twin peak and the cycle in
Figure~\ref{fig:2cubes} are the two combinatorially different
PUSOs of the 2-cube.  The 3- cube has $19$ combinatorially
different USOs~\cite{StiWat}, but only two combinatorially different
PUSOs, see Figure~\ref{fig:3pusos} together with
Corollary~\ref{cor:3pusos} below.

\begin{figure}[htb]
\begin{center}
\includegraphics[width=0.65\textwidth]{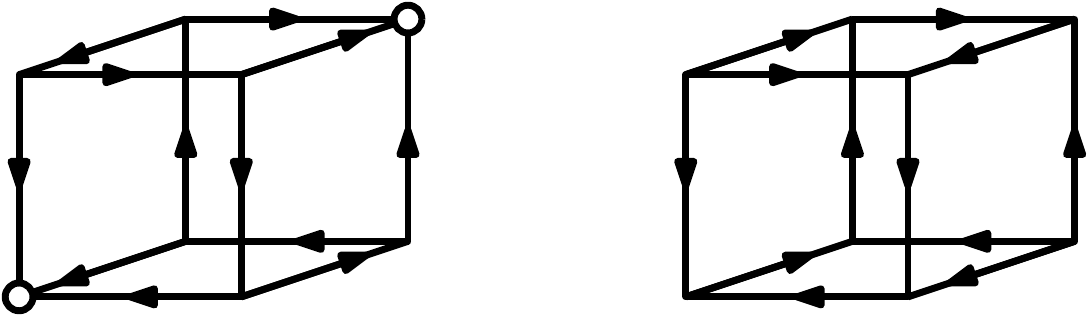}
\end{center}
\caption{The two combinatorially different PUSOs of the 3-cube}
\label{fig:3pusos}
\end{figure}

We let $\uso(n)$ and $\puso(n)$ denote the number of USOs and PUSOs of
the standard $n$-cube. We have $\uso(0)=1,\uso(1)=2$ as well as
$\uso(2)=12$ (4 eyes and 8 bows). Moreover, $\puso(0)=\puso(1)=0$ and
$\puso(2)=4$ (2 twin peaks, 2 cycles).  

\section{Outmaps of (Pseudo) USOs}\label{sec:outmaps}

Outmaps of USOs have a simple characterization~\cite[Lemma
2.3]{szabo2001unique}:
$\outmap:\vert{\Cube}\rightarrow 2^{\carr{\Cube}}$ is the outmap of a
  USO of $\Cube$ if and only if
\begin{equation}\label{eq:char}
(\outmap(U)\oplus \outmap(V))\cap (U\oplus V) \neq \emptyset
\end{equation}
holds for all pairs of distinct vertices $U,V\in\vert{\Cube}$. This condition
means the following: within the face
$\Cube^{[U\cap V, U\cup V]}$ \emph{spanned} by $U$ and $V$, there is a
coordinate that is outgoing for exactly one of the two vertices. In particular,
any two distinct vertices have different outmap values, so
$\outmap$ is injective and hence bijective. 

This characterization implicitly makes a more general statement: for
every face $\Face$, orientation $\Face_{\outmap}$ is a USO if and only
if (\ref{eq:char}) holds for all pairs of distinct vertices
$U,V\in\Face$. The reason is that the validity of (\ref{eq:char}) only
depends on the behavior of $\outmap$ within the face spanned by $U$ and
$V$. Formally, for $U,V\in\Face$, (\ref{eq:char}) is equivalent to the
USO-characterizing condition
$(\outmap_{\Face}(U)\oplus \outmap_{\Face}(V))\cap (U\oplus
V)\neq\emptyset$
for the orientation $\Face_{\outmap_{\Face}}=\Face_{\outmap}$.

\begin{lemma}\label{lem:strongchar}
  Let $\Cube$ be a cube,
  $\outmap: \vert{\Cube} \rightarrow 2^{\carr{\Cube}}$,
  $\Face$ a face of $\Cube$. Then $\Face_{\outmap}$ is a USO if and
  only if
\[
(\outmap(U)\oplus \outmap(V))\cap (U\oplus V) \neq \emptyset
\]
holds for all pairs of distinct vertices $U,V\in\vert{\Face}$. In this
case, the outmap $\outmap_{\Face}$ of $\Face_{\outmap}$ is bijective.
\end{lemma} 

As a consequence, outmaps of PUSOs can be characterized as follows:
(\ref{eq:char}) holds for all pairs of non-antipodal vertices $U,V$
(which always span a proper face), but fails for some pair
$U, V=\carr{\Cube}\setminus U$ of antipodal vertices. As the validity
of (\ref{eq:char}) is invariant under replacing all outmap values
$\outmap(V)$ with $\outmap'(V)=\outmap(V)\oplus R$ for some fixed
$R\subseteq\carr{\Cube}$, we immediately obtain that PUSOs (as well as
USOs~\cite[Lemma 2.1]{szabo2001unique}) are closed under
\emph{flipping coordinates} (reversing all edges along some subset of
the coordinates).

\begin{lemma}
    \label{lem:puso_flip}
    Let $\Cube$ be a cube,
    $\outmap: \vert{\Cube} \rightarrow 2^{\carr{\Cube}}$, $\Face$ a
    face of $\Cube$.  Suppose that $\Face_{\outmap}$ is a PUSO and
    $R\subseteq \carr{\Cube}$. Consider the \emph{$R$-flipped}
    orientation $\Cube_{\outmap'}$ induced by the outmap
    \begin{equation*}
        \outmap'(V) := \outmap(V) \oplus R, \quad \forall V\in \vert{\Cube}.
    \end{equation*}
    Then $\Face_{\outmap'}$ is a PUSO as well.
\end{lemma}

Using this, we can show that in a PUSO, (\ref{eq:char}) must actually
fail on \emph{all} pairs of antipodal vertices, not just on some pair,
and this is the key to the strong structural properties of PUSOs.

\begin{theorem}[PUSO characterization]\label{thm:char}
  Let $\Cube$ be a cube of dimension at least $2$,
  $\outmap:\vert{\Cube}\rightarrow 2^{\carr{\Cube}}$, $\Face$ a face
  of $\Cube$. Then $\Face_{\outmap}$ is a PUSO
  if and only if
\begin{itemize}
\item[(i)] condition (\ref{eq:char}) holds for
  all $U,V\in \vert{\Face}, V\neq U, \carr{\Face}\setminus U$ (pairs
  of distinct, non-antipodal vertices in $\Face$), and
\item[(ii)] condition (\ref{eq:char}) fails for all
  $U,V\in \vert{\Face}, V=\carr{\Face}\setminus U$ (pairs of antipodal
  vertices in $\Face$).
\end{itemize}
\end{theorem}

\begin{proof} In view of the above discussion, it only remains to show
  that (ii) holds if $\Face_{\outmap}$ is a PUSO. Let
  $U\in\vert{\Face}$.  Applying Lemma~\ref{lem:puso_flip} with
  $R=\outmap(U)$ does not affect the validity of (\ref{eq:char}), so we may assume
  w.l.o.g.\ that $\outmap(U)=\emptyset$, hence $U$ is a sink in
  $\Face_{\outmap}$. For a non-antipodal $W\in\vert{\Face}$, (i)
  implies the existence of some
  $i\in\outmap(W)\cap(U\oplus W)\subseteq
  \outmap(W)\cap\carr{\Face}=\outmap_{\Face}(W)$,
  hence such a $W$ is not a sink in $\Face_{\outmap}$. But then
  $V=\carr{\Face}\setminus U$ must be a second sink in
  $\Face_{\outmap}$, because PUSO $\Face_{\outmap}$ does not have a
  unique sink. This in turn implies that (\ref{eq:char}) fails for
  $U,V=\carr{\Face}\setminus U$.
\end{proof}

\begin{corollary} \label{cor:puso_antipodals} Let $\Cube_{\outmap}$ be
  a PUSO with outmap $\outmap$.
\begin{itemize}
\item[(i)] Any two antipodal vertices $U,V=\carr{\Cube}\setminus U$
  have the same outmap value, $\outmap(U)=\outmap(V)$. 
\item[(ii)] $\Cube_{\outmap}$ either has no sink, or exactly two sinks.
\end{itemize}
\end{corollary}
\begin{proof} For antipodal vertices, $U\oplus V=\carr{\Cube}$, so 
  $(\outmap(U)\oplus \outmap(V))\cap (U\oplus V) = \emptyset$ is
  equivalent to $\outmap(U)=\outmap(V)$. In particular,
  the number of sinks is even but cannot exceed $2$, as otherwise,
  there would be two non-antipodal sinks; the proper face they
  span would then have more than one sink, a contradiction.
\end{proof}

We can use the characterization of Theorem~\ref{thm:char} to show that
PUSOs exist in every dimension $n\geq 2$.

\begin{lemma}[PUSO Existence]
  \label{puso_existence} Let $n \geq 2$, $\Cube$ the standard $n$-cube
  and $\pi:[n]\rightarrow[n]$ a permutation consisting of a single
  $n$-cycle. Consider the function $\outmap: 2^{[n]} \mapsto 2^{[n]}$
  defined by
    \[
    \outmap(V) = \{i\in [n]: |V\cap \{i, \pi(i)\}|=1\}, \quad \forall 
    V\subseteq[n].
    \]
    Then $\Cube_{\outmap}$ is a PUSO. 
\end{lemma}

\begin{proof}
  According to Theorem~\ref{thm:char}, we need to show that condition
  (\ref{eq:char}) fails for all pairs of antipodal vertices, but that
  it holds for all pairs of distinct vertices that are not
  antipodal. 

  We first consider two antipodal vertices $U$ and $V=[n]\setminus U$
  in which case we get $\outmap(U)=\outmap(V)$, so (\ref{eq:char})
  fails. If $U$ and $V$ are distinct and not antipodal, there is some
  coordinate in which $U$ and $V$ differ, \emph{and} some coordinate
  in which $U$ and $V$ agree. Hence, if we traverse the $n$-cycle
  $(1,\pi(1),\pi(\pi(1)),\ldots)$, we eventually find two consecutive
  elements $i,\pi(i)$ such that $U$ and $V$ differ in coordinate $i$
  but agree in coordinate $\pi(i)$, meaning that
  $i\in (\outmap(U)\oplus \outmap(V))\cap (U\oplus V)$, so
  (\ref{eq:char}) holds.
\end{proof}

We conclude this section with another consequence of
Theorem~\ref{thm:char} showing that PUSOs have a parity.

\begin{lemma}
    \label{lem:puso_parity}
    Let $\Cube_{\outmap}$ be a PUSO with outmap $\outmap$.
    Then the outmap values of all vertices have the same parity, that is
   \[
   |\outmap(U) \oplus \outmap(V)| = 0 \mod 2, \quad \forall U,V\in\vert{\Cube}.
   \]
\end{lemma}

We call the number $|\outmap(\emptyset)|\mod 2$ the \emph{parity} of
$\Cube_{\outmap}$. By Corollary~\ref{cor:puso_antipodals}, a PUSO of
even parity has two sinks, a PUSO of odd parity has none.

\begin{proof}
  We first show that the outmap valus of any two distinct non-antipodal
  vertices $U$ and $V$ differ in at least two coordinates. Let $V'$ be
  the antipodal vertex of $V$. As $U$ is neither antipodal to $V$ nor
  to $V'$, Theorem~\ref{thm:char} along with
  $\outmap(V)=\outmap(V')$
  (Corollary~\ref{cor:puso_antipodals}) yields
    \begin{eqnarray*}
      (\outmap(U)\oplus \outmap(V))\cap (U\oplus V) &\neq& \emptyset, \\
      (\outmap(U)\oplus \outmap(V))\cap (U\oplus V') &\neq& \emptyset.
    \end{eqnarray*}
    Since $U \oplus V$ is disjoint from $U \oplus V'$,
    $\outmap(U)\oplus \outmap(V)$ contains at
    least two coordinates.

    Now we can prove the actual statement. Let $I$ be the image of
    $\outmap$,
    $I := \{\outmap(V): V\in\vert{\Cube}\}\subseteq
    \Cube'=\Cube^{\carr{\Cube}}$.
    We have $|I|\geq 2^{n-1}$, because by Lemma~\ref{lem:strongchar},
    $\outmap_{\Face}$ is bijective (and hence $\outmap$ is injective)
    on each facet $\Face$ of $\Cube$. On the other hand, $I$ forms an
    independent set in the cube $\Cube'$, as any two distinct outmap
    values differ in at least two coordinates; The statement follows,
    since the only independent sets of size at least $2^{n-1}$ in an
    $n$-cube are formed by all vertices of fixed parity.
\end{proof}

\section{Recognizing (Pseudo) USOs}\label{sec:recognition}
Before we dive deeper into the structure of PUSOs in the next section,
we want to present a simple algorithmic consequence of the PUSO
characterization provided by Theorem~\ref{thm:char}.

Suppose that $\Cube$ is an $n$-cube, and that an outmap
$\outmap:\vert{\Cube}\rightarrow 2^{\carr{\Cube}}$ is succinctly given
by a Boolean circuit of polynomial size in $n$. Then it is
\coNP-complete to decide whether $\Cube_{\outmap}$ is a
USO~\cite{gaertner2015complexity}.\footnote{In fact, it is already
  \coNP-complete to decide whether $\Cube_{\outmap}$ is an
  orientation.}  $\coNP$-membership is easy: every non-USO has a
certificate in the form of two vertices that fail to satisfy
(\ref{eq:char}). Finding two such vertices is hard, though. For given
vertices $U$ and $V$, let us call the computation of
$(\outmap(U)\oplus \outmap(V))\cap (U\oplus V)$ a \emph{pair
  evaluation}. Then, the obvious algorithm needs $\Theta(4^n)$
pair evaluations. Using Theorem~\ref{thm:char}, we can improve on
this.

\begin{theorem}[Faster USO recognition]\label{thm:check}
  Let $\Cube$ be an $n$-cube,
  $\outmap:\vert{\Cube}\rightarrow 2^{\carr{\Cube}}$. Using $O(3^n)$
  pair evaluations, we can check whether $\Cube_{\outmap}$ is a USO.
\end{theorem}
\begin{proof}
  For every face $\Face$ of dimension at least $1$ (there are $3^n-2^n$ of
  them), we perform a pair evaluation with an arbitrary pair of
  antipodal vertices $U,V=\carr{\Face}\setminus U$. We output that
  $\Cube_{\outmap}$ is a USO if and only if all these pair evaluations
  succeed (meaning that they return nonempty sets). 

  We need to argue that this is correct. Indeed, if $\Cube_{\outmap}$
  is a USO, all pair evaluations succeed by Lemma~\ref{lem:strongchar}.
  If $\Cube_{\outmap}$ is not a USO, it is either not
  an orientation (so the pair evaluation in some 1-face fails), or it
  contains a PUSO $\Face_{\outmap}$ in which case the pair evaluation
  in $\Face$ fails by Theorem~\ref{thm:char}.
\end{proof}

Using the same algorithm, we can also check whether $\Cube_{\outmap}$
is a PUSO. Which is the case if and only if the pair evaluation
succeeds on every face except $\Cube$ itself.

\section{Border Unique Sink Orientations}\label{sec:border}
Lemma~\ref{lem:puso_parity} already implies that not every USO can
occur as a facet of a PUSO. For example, let us assume that an eye
(Figure~\ref{fig:2cubes}) appears as a facet of a 3-dimensional
PUSO. Then, Corollary~\ref{cor:puso_antipodals} (i) completely
determines the orientation in the opposite facet: we get a ``mirror
orientation'' in which antipodal vertices have traded outgoing
coordinates; see Figure~\ref{fig:Noeye}.

\begin{figure}[htb]
\begin{center}
\includegraphics[width=0.3\textwidth]{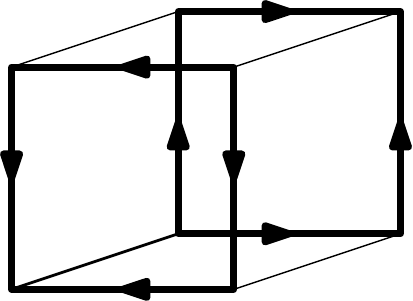}
\end{center}
\caption{An eye is not a facet of a PUSO}
\label{fig:Noeye}
\end{figure}

But now, every edge between the two facets connects two vertices with
the same outmap parity within their facets, and no matter how we
orient the edge, the two vertices will receive different global outmap
parities. Hence, the resulting orientation cannot be a PUSO by
Lemma~\ref{lem:puso_parity}.

It therefore makes sense to study the class of \emph{border USOs},
the USOs that appear as facets of PUSOs.

\begin{definition}[Border USO]
  A \emph{border USO} is a USO that is a facet of some PUSO.
\end{definition}

If the border USO lives on cube $\Face=\Cube^{[A,B]}$, the PUSO may
live on $\Cube^{[A,B\cup\{n\}]}$ ($n\notin\carr{\Face}$ a new
coordinate), or on $\Cube^{[A\setminus\{n\},B]}$ ($n\in A$), but these
cases lead to combinatorially equivalent situations. We will always think
about extending border USOs by adding a new coordinate.

In this section, we characterize border USOs. We already know that
antipodal vertices must have outmap values of different parities; a
generalization of this yields a sufficient condition: if the outmap
values of distinct vertices $U,V$ agree outside of the face spanned
by $U$ and $V$, then the two outmap values must have different
parities.

\begin{theorem}[Border USO characterization]
    \label{thm:pusof_charact}
    Let $\Face_{\outmapp}$ be a USO with outmap
    $\outmapp$. $\Face_{\outmapp}$ is a border USO if and only if the
    following condition holds for all pairs of distinct vertices
    $U,V\in\vert{\Face}$:
    \begin{equation}
        \label{pusof_thm_eq}
        \outmapp(U)\oplus \outmapp(V) \subseteq U\oplus V \quad
        \Rightarrow \quad |\outmapp(U)\oplus \outmapp(V)| = 1 \mod 2.
    \end{equation}
\end{theorem}

A preparatory step will be to generalize the insight gained from the
case of the eye above and show that a USO can be extended to a PUSO
of one dimension higher in at most two canonical ways---exactly
two if the USO is actually border.  
 
\begin{lemma}\label{lem:outmap01}
  Let $\Face$ be a facet of $\Cube$,
  $\carr{\Cube}\setminus\carr{\Face}=\{n\}$, 
  and let $\Face_{\outmapp}$ be a USO with outmap $\outmapp$.

\begin{itemize}
\item[(i)] There are at most
  two outmaps $\outmap:\vert{\Cube}\rightarrow 2^{\carr{\Cube}}$ such that $\Cube_{\outmap}$ is a PUSO with
  $\Face_{\outmap}=\Face_{\outmapp}$. Specifically, these are
  $\outmap_i,i=0,1$, with
\begin{equation}\label{eq:outmap01}
\outmap_i (V) = \left\{\begin{array}{ll} 
\outmapp(V), & V\in\vert{\Face},~ |\outmapp(V)| = i \mod 2, \\
\outmapp(V) \cup\{n\}, & V\in\vert{\Face},~ |\outmapp(V)| \neq i \mod 2, \\
\outmap_i(\carr{C}\setminus V), &V\notin\vert{\Face}.
                       \end{array}\right.
                   \end{equation}
\item[(ii)] If $\Face_{\outmapp}$ is a border USO, both
  $\Cube_{\outmap_0}$ and $\Cube_{\outmap_1}$ are PUSOs.
\item[(iii)] If $\Cube_{\outmap_i}$ is a PUSO for some $i\in\{0,1\}$,
  then $\Cube_{\outmap_{1-i}}$ is a PUSO as well, and
  $\Face_{\outmapp}$ is a border USO.
                 \end{itemize}
               \end{lemma}

\begin{proof}
  Only for $\outmap=\outmap_i,i=0,1$, we obtain
  $\Face_{\outmap}=\Face_{\outmapp}$ and satisfy the necessary
  conditions of Corollary~\ref{cor:puso_antipodals} (pairs of
  antipodal vertices have the same outmap values in a PUSO), and of
  Lemma~\ref{lem:puso_parity} (all outmap values have the same parity
  in a PUSO). Hence, $\Cube_{\outmap_0}$ and $\Cube_{\outmap_1}$ are
  the only candidates for PUSOs extending $\Face_{\outmapp}$. This
  yields (i). If $\Face_{\outmapp}$ is a border USO, one of the
  candidates is a PUSO by definition; as the other one results from it
  by just flipping coordinate $n$, it is also a PUSO by
  Lemma~\ref{lem:puso_flip}. Part (ii) follows. For part (iii), we use
  that $\Face_{\outmapp}$ is a facet of $\Cube_{\outmap_i}$, $i=0,1$,
  so as before, if one of the latter is a PUSO, then both are, and
  $\Face_{\outmapp}$ is a border USO by definition.
\end{proof}

\begin{corollary}\label{cor:3pusos}
There are 2 combinatorially different PUSOs of the 3-cube (depicted in
Figure~\ref{fig:3pusos}). 
\end{corollary}
\begin{proof}
  We have argued above that an eye cannot be extended to a PUSO, so
  let us try to extend a bow (the front facet in
  Figure~\ref{fig:3pusos}). The figure shows the two candidates for
  PUSOs provided by Lemma~\ref{lem:outmap01}. Both happen to be PUSOs,
  so starting from the single combinatorial type of 2-dimensional
  border USOs, we arrive at the two combinatorial types of
  3-dimensional PUSOs.
\end{proof}

Concluding this section, we prove the advertised characterization of border USOs.

\begin{proof} {[Theorem~\ref{thm:pusof_charact}]} Let $\Cube$ be a cube with facet $\Face$,
  $\carr{\Cube}\setminus\carr{\Face}=\{n\}$. We show that condition
  (\ref{pusof_thm_eq}) fails for some pair of distinct vertices
  $U,V\in\vert{\Face}$ if and only if $\Cube_{\outmap_0}$ is not a
  PUSO, with $\outmap_0$ as in (\ref{eq:outmap01}). By
  Lemma~\ref{lem:outmap01}, this is equivalent to $\Face_{\outmapp}$
  not being a border USO.

  Suppose first that there are distinct $U,V\in\vert{\Face}$ such that 
  $\outmapp(U)\oplus \outmapp(V) \subseteq U\oplus V$ and
  $|\outmapp(U)\oplus \outmapp(V)| =0 \mod 2$, meaning that $U$ and
  $V$ have the same outmap parity. By definition of $\outmap_0$, we
  then get
  \[
  \outmapp(U)\oplus \outmapp(V)=\outmap_0(U)\oplus \outmap_0(V) =
  \outmap_0(U)\oplus \outmap_0(V') \subseteq U\oplus V,
\]
where $V'=\carr{\Cube}\setminus V$ is antipodal to $V$ in
$\Cube$. Moreover, as $U\oplus V$ is also antipodal to $U\oplus V'$,
the inclusion
$\outmap_0(U)\oplus \outmap_0(V)=\outmap_0(U)\oplus
\outmap_0(V')\subseteq U\oplus V$ is equivalent to
\begin{equation}\label{eq:UV'}
(\outmap_0(U)\oplus \outmap_0(V'))\cap(U\oplus V')=\emptyset.
\end{equation}
Since $U,V$ are distinct and non-antipodal (in $\Cube$), $U,V'$ are
therefore distinct non-antipodal vertices that fail to satisfy
Theorem~\ref{thm:char}~(i), so $\Cube_{\outmap_0}$ is not a PUSO.

For the other direction, we play the movie backwards. Suppose that
$\Cube_{\outmap_0}$ is not a PUSO. As pairs of antipodal vertices
comply with Theorem~\ref{thm:char}~(ii) by definition of $\outmap_0$,
there must be distinct and non-antipodal vertices $U,V'$ with the
offending property (\ref{eq:UV'}). Moreover, as $\outmap_0$ induces
USOs on both $\Face$ (where we have $\Face_{\outmapp}$) and its
opposite facet $\Face'$ (where we have a mirror image of
$\Face_{\outmapp}$), Lemma~\ref{lem:strongchar} implies that $U$ and
$V'$ cannot both be in $\Face$, or in $\Face'$. W.l.o.g.\ assume that
$U\in\vert{\Face}, V'\in\vert{\Face'}$, and let $V\in\vert{\Face}$ be
antipodal to $V'$.  Then, as before, (\ref{eq:UV'}) is equivalent to
the inclusion
$\outmap_0(U)\oplus \outmap_0(V)=\outmap_0(U)\oplus
\outmap_0(V')\subseteq U\oplus V$. In particular, $\outmap_0(U)$ and $\outmap_0(V)$ must agree in
coordinate $n$ which in turn implies
\[\outmapp(U)\oplus \outmapp(V)=\outmap_0(U)\oplus \outmap_0(V)
\subseteq U\oplus V,\]
and since $|\outmap_0(U)\oplus \outmap_0(V)|=0\mod 2$ by definition of
$\outmap_0$, we have found two distinct vertices $U,V\in\vert{\Face}$
that fail to satisfy (\ref{pusof_thm_eq}).
\end{proof}

For an example of a 3-dimensional border USO, see
Figure~\ref{fig:3pusof}. In particular, we see that faces of border
USOs are not necessarily border USOs: an eye cannot be a 2-dimensional
border USO (Figure~\ref{fig:Noeye}), but it may appear in a facet
$\Face$ of a 3-dimensional border USO (for example, the bottom facet
in Figure~\ref{fig:3pusof}), since the incident edges along the third
coordinate can be chosen such that (\ref{pusof_thm_eq}) does not
impose any condition on the USO in $\Face$.

\begin{figure}[htb]
\begin{center}
\includegraphics[width=0.3\textwidth]{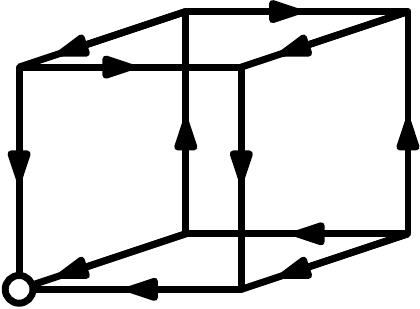}
\end{center}
\caption{A border USO with eyes (non-border USOs) in the bottom and the top facet}
\label{fig:3pusof}
\end{figure}

Let $\border(n)$ denote the number of border USOs of the standard
$n$-cube. By Lemma~\ref{lem:outmap01}, 
\begin{equation}\label{eq:pusoborder}
\puso(n)=2\border(n-1), \quad n\geq 2.
\end{equation}

\section{Odd Unique Sink Orientations}\label{sec:odd}
By (\ref{eq:pusoborder}), counting PUSOs boils down to counting border
USOs. However, as faces of border USOs are not necessarily border USOs
(see the example of Figure~\ref{fig:3pusof}), it will be easier to
work in a dual setting where we get a class of USOs that is closed
under taking faces.

\begin{lemma}\label{lem:duality} Let $\Cube_{\outmap}$ be a USO of
  $\Cube=\Cube^{B}$ with outmap $\outmap$. Then
  $\Cube_{\outmap^{-1}}$ is a USO as well, the \emph{dual} of
  $\Cube_{\outmap}$.
\end{lemma}
\begin{proof} We use the USO characterization of Lemma~\ref{lem:strongchar}. Since
  $\Cube_{\outmap}$ is a USO, $\outmap:2^B\rightarrow 2^B$ is
  bijective to begin with, so $\outmap^{-1}$ exists. Now let
  $U',V'\in\vert{\Cube}$, $U'\neq V'$ and define
  $U:=\outmap^{-1}(U')\neq\outmap^{-1}(V')=:V$. Then we have
\[
(\outmap^{-1}(U')\oplus \outmap^{-1}(V'))\cap (U'\cap V') =
(U\oplus V) \cap (\outmap(U)\oplus\outmap(V) \neq \emptyset,
\]
since $\Cube_{\outmap}$ is a USO. Hence, $\Cube_{\outmap^{-1}}$ is a USO 
as well.
\end{proof}

\begin{definition}[Odd USO]\label{def:odd}
    An \emph{odd USO} is a USO that is dual to a border USO.
\end{definition}

Figure~\ref{fig:duality} shows an
example of the duality with the following outmaps:

\[
\begin{array}{c|c|c|c|c|c|c|c|c|c}
V & \emptyset & \{1\} & \{2\} & \{3\} & \{1,2\} & \{1,3\} & \{2,3\} &
                                                                      \{1,2,3\}
  & \outmap^{-1}(V')    \\ \hline
\outmap(V) &  \emptyset & \{1\} & \{1,2\} & \{2,3\} & \{2\} &
                                                              \{1,2,3\}
                                                          & \{1,3\} &
                                                                      \{3\}
  & V'                                     
\end{array}
\]

\begin{figure}[htb]
\begin{center}
\includegraphics[width=0.9\textwidth]{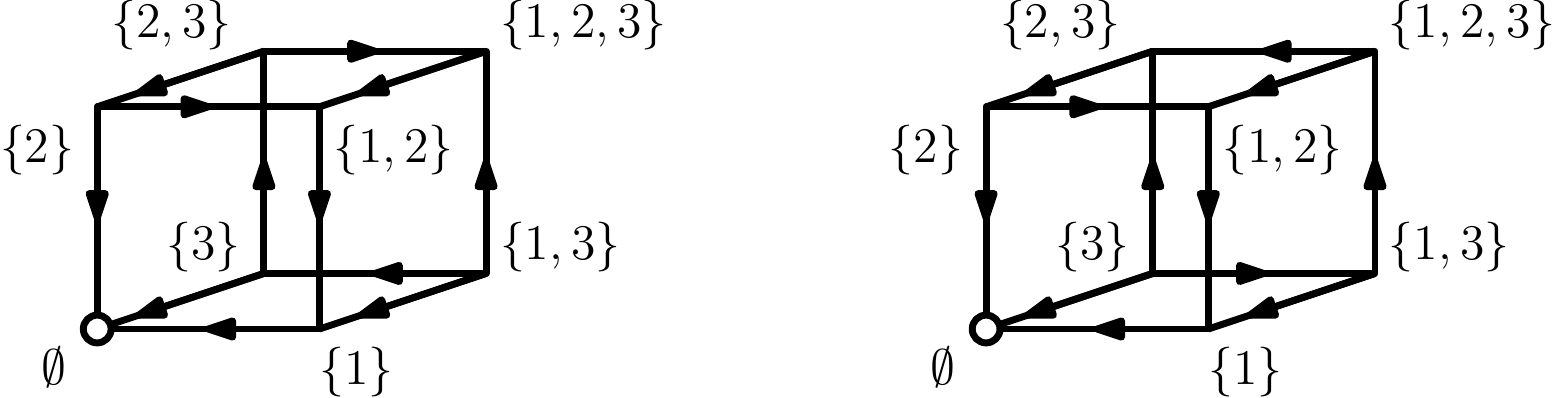}
\end{center}
\caption{The border USO $\Cube_{\outmap}$ of Figure~\ref{fig:3pusof} (left) and its dual odd USO $\Cube_{\outmap^{-1}}$
  (right). The labels denote the vertices.}
\label{fig:duality}
\end{figure}

A characterization of odd USOs now follows from
Theorem~\ref{thm:pusof_charact} by swapping the roles of vertices and
outmaps; the proof follows the same scheme as the one of
Lemma~\ref{lem:duality} and is omitted.

\begin{theorem}[Odd USO characterization]
  \label{thm:odd_charact}
    Let $\Cube_{\outmap}$ be a USO with outmap $\outmap$. $\Cube_{\outmap}$ is an odd USO 
    if and only if the following condition holds for all pairs of
    distinct vertices $U,V\in\vert{\Cube}$:
    \begin{equation}
        \label{eq:odd_thm_eq}
        U\oplus V \subseteq \outmap(U)\oplus \outmap(V) \quad
        \Rightarrow \quad |U\oplus V| = 1 \mod 2.
    \end{equation}
\end{theorem}

In words, if the outmap values of two distinct vertices $U,V$ differ
in all coordinates within the face spanned by $U$ and $V$, then $U$
and $V$ are of odd Hamming distance.\footnote{Hamming distance is
  defined for two bit vectors, but we can also define it for two sets
  in the obvious way as the size of their symmetric difference.} As
this property also holds for any two distinct vertices within a face
$\Face$, this implies the following.

\begin{corollary} \label{cor:odd} Let $C_{\outmap}$ be an odd USO, $\Face$ a face of $\Cube$.
\begin{itemize}
\item[(i)] $\Face_{\outmap}$ is an odd USO.
\item[(ii)] If $\dim(\Face)=2$, $\Face_{\outmap}$ is a bow.
  \end{itemize}
\end{corollary}

Indeed, as source and sink of an eye violate (\ref{eq:odd_thm_eq}), all
2-faces of odd USOs are bows. To make the global structure of odd USOs
more transparent, we develop an alternative view on them in terms of
\emph{caps} that can be considered as ``higher-dimensional bows''.

\begin{definition}[Cap]
  Let $\Cube_{\outmap}$ be an orientation with bijective outmap
  $\outmap$. For $W\in\vert{\Cube}$, let $\overline{W}\in\vert{\Cube}$
  be the unique \emph{complementary} vertex, the one whose outmap
  value is antipodal to $\outmap(W)$; formally,
  $\outmap(W)\oplus \outmap(\overline{W})=\carr{\Cube}$.
  $\Cube_{\outmap}$ is called a \emph{cap} if
\[
|W\oplus\overline{W}|=1\mod 2, \quad \forall W\in\vert{\Cube}.
\]
\end{definition}

Figure~\ref{fig:oddface} illustrates this notion on three examples.

\begin{figure}[htb]
\begin{center}
\includegraphics[width=0.9\textwidth]{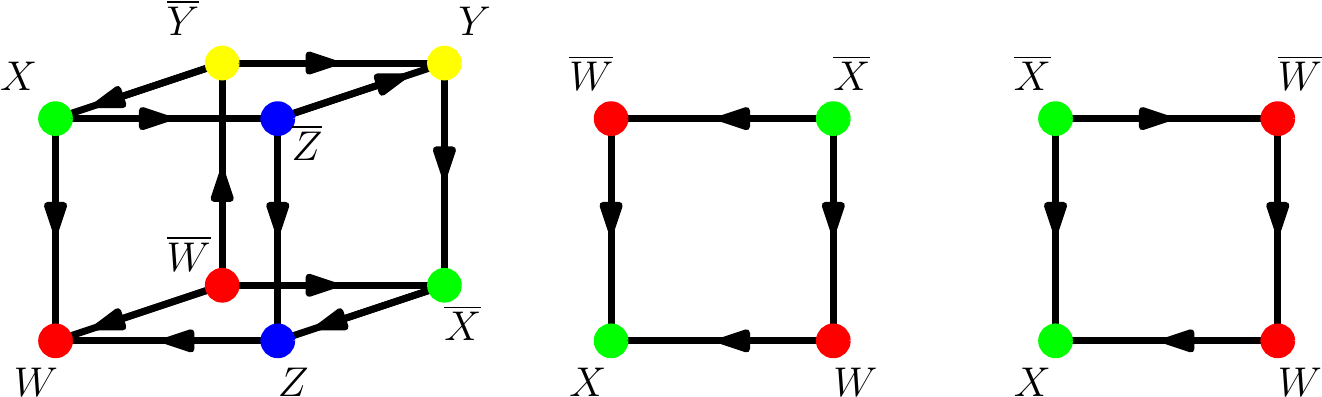}
\end{center}
\caption{A 3-dimensional cap: complementary vertices (vertices that
  differ in all outgoing coordinates) have odd Hamming distance
  (left); the eye is not a cap (middle); the bow is a cap (right).}
\label{fig:oddface}
\end{figure}

\begin{lemma}\label{lem:oddchar}
  Let $\Cube_{\outmap}$ be an orientation with outmap $\outmap$. $\Cube_{\outmap}$ is an odd
  USO if and only if all its faces are caps.
\end{lemma}

\begin{proof}
  If all faces are caps, their outmaps are bijective, meaning that
  all faces have unique sinks. So $\Cube_{\outmap}$ is a USO. It is odd,
  since the characterizing property (\ref{eq:odd_thm_eq}) follows for all
  distinct $U,V$ via the cap spanned by $U$ and $V$. 

  Now suppose that $\Cube_{\outmap}$ is an odd USO. Then every face
  $\Face$ has a bijective outmap to begin with, by
  Lemma~\ref{lem:strongchar}; to show that $\Face$ is a cap, consider
  any two complementary vertices $W,\overline{W}$ in $\Face$. As $W$
  and $\overline{W}$ are in particular complementary in the face that
  they span, they have odd Hamming distance by
  (\ref{eq:odd_thm_eq}).
\end{proof}

There is a ``canonical'' odd USO of the standard $n$-cube in which the
Hamming distances of complementary vertices are not only odd, but in
fact always equal to $1$. This orientation is known as the
\emph{Klee-Minty cube}, as it captures the combinatorial structure of
the linear program that Klee and Minty used in 1972 to show for the first time
that the simplex algorithm may take exponential time~\cite{KM}.

The $n$-dimensional Klee-Minty cube can be defined inductively:
$\KM^{[n]}$ is obtained from $\KM^{[n-1]}$ by embedding an
$[n-1]$-flipped copy of $\KM^{[n-1]}$ into the opposite facet
$\Cube^{[\{n\},[n]]}$, with all connecting edges oriented towards
$\Cube^{[n-1]}$; the resulting USO contains a directed Hamiltonian
path; see Figure~\ref{fig:KM}. As a direct consequence of the
construction, $\KM^{[n]}$ is a cap: complementary vertices are neighbors
along coordinate $n$. Moreover, it is easy to see that each $k$-face
is combinatorially equivalent to $\KM^{[k]}$, hence all faces are caps,
so $\KM^{[n]}$ is an odd USO.

\begin{figure}[htb]
\begin{center}
\includegraphics[width=\textwidth]{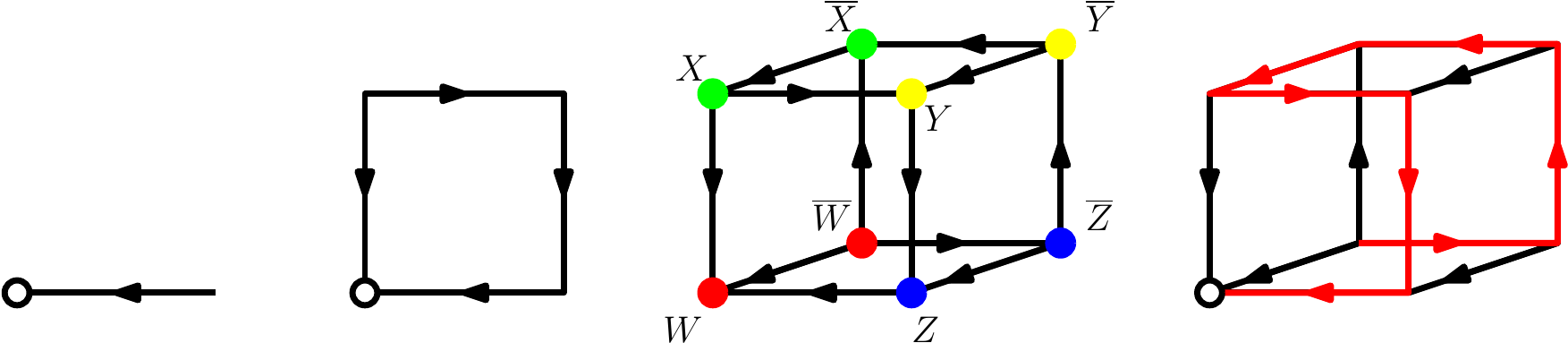}
\end{center}
\caption{The Klee-Minty cubes $\KM^{[1]},
  \KM^{[2]},\KM^{[3]}$ (with pairs of complementary vertices and directed Hamiltonian path)}
\label{fig:KM}
\end{figure}

Next, we do this more formally, as we will need the Klee-Minty
cube as a starting point for generating many odd USOs. 

\begin{lemma} \label{lem:km} Consider the standard $n$-cube $\Cube$ and the outmap
  $\outmap:2^{[n]}\rightarrow 2^{[n]}$ with
\[
\outmap (V) = \{j\in[n]: |V\cap\{j,j+1,\ldots,n\}|=1\mod 2\}, \quad
\forall V\subseteq [n].
\]
Then $\KM^{[n]}:=\Cube_{\outmap}$ is an odd USO that satisfies
\begin{equation}\label{eq:vertexflip}
\outmap(W) \oplus \outmap(W\oplus\{i\}) = [i], \quad \forall i\in[n].
\end{equation}
for each vertex $W$.
\end{lemma}

In particular, for all $i$ and $W\in\Cube^{[i-1]}$, $W$ and
$W\cup\{i\}$ are complementary in $\Cube^{[i]}$, so we recover the
above inductive view of the Klee-Minty cube.

\begin{proof}
We first show that 
\begin{equation}\label{eq:kmout}
\outmap(U)\oplus\outmap(V) = \outmap(U\oplus V), \quad \forall U,V\subseteq[n].
\end{equation}
Indeed, $j\in \outmap(U)\oplus\outmap(V)$ is equivalent to
$U\cap\{j,j+1,\ldots,n\}$ and $V\cap\{j,j+1,\ldots,n\}$ having
different parities, which by (\ref{eq:symdiff}) is equivalent to
$(U\oplus V)\cap\{j,j+1,\ldots,n\}$ having odd parity, meaning that $j\in\outmap(U\oplus V)$.

Since $\outmap(U)\oplus\outmap(V)=\outmap(U\oplus V)$ contains the
largest element of $U\oplus V$, (\ref{eq:char}) holds for all pairs of
distinct vertices, so $\Cube_{\outmap}$ is a USO. Condition
(\ref{eq:vertexflip}) follows from
\[
\outmap(W) \oplus \outmap(W\oplus\{i\}) =  \outmap(\{i\}) = [i].
\]
To show that $\Cube_{\outmap}$ is odd, we verify condition
(\ref{eq:odd_thm_eq}) of Theorem~\ref{thm:odd_charact}. Suppose that
$U\oplus V \subseteq \outmap(U)\oplus \outmap(V) = \outmap(U\oplus
V)$ for two distinct vertices. Since $\outmap(U\oplus V)$ does not
contain the second-largest element of $U\oplus V$, the former
inclusion can only hold if there is no such second-largest element,
i.e. $U$ and $V$ have (odd) Hamming distance $1$. 
\end{proof}

The Klee-Minty cube has a quite special property: \emph{complementing}
any vertex (reversing all its incident edges) yields another odd
USO.\footnote{In general, the operation of complementing a vertex will
  destroy the USO property.} Even more is true: any set of vertices
with disjoint neighborhoods can be complemented simultaneously. Thus,
if we select a set of $N$ vertices with pairwise Hamming distance at
least $3$, we get $2^N$ different odd USOs. We will use this in the
next section to get a lower bound on the number of odd USOs. The
following lemma is our main workhorse.

\begin{lemma}\label{lem:vertexflip}
  Let $\Cube_{\outmap}$ be an odd USO of the standard $n$-cube with
  outmap $\outmap$, $W\in\vert{\Cube}$ a vertex satisfying condition
  (\ref{eq:vertexflip}):
\[
\outmap(W) \oplus \outmap(W\oplus\{i\}) = [i], \quad \forall i\in[n].
\]
Let $\Cube_{\outmap'}$ be the orientation resulting from complementing
(reversing all edges incident to) $W$. Formally, 
\begin{equation}\label{eq:compl}
\begin{array}{lclcl}
\outmap'(W) &=& \outmap(W) &\oplus& [n], \\
\outmap'(W\oplus\{i\}) &=& \outmap (W\oplus\{i\}) &\oplus& \{i\}, \quad
  i=1,\ldots,n,
\end{array}
\end{equation}
and $\outmap'(V)=\outmap(V)$ for  all other vertices. Then $\Cube_{\outmap'}$ is an
odd USO as well.
\end{lemma}

\begin{proof}
  We first show that every face $\Face_{\outmap'}$ has a unique sink,
  so that $\outmap'$ is a USO. If $W\not\in\vert{\Face}$, then
  $\Face_{\outmap'}=\Face_{\outmap}$, so there is nothing to
  show. If $W\in \vert{\Face}$, let
  $\carr{\Face}=\{i_1,i_2,\ldots,i_k\}$, $i_1<i_2<\cdots<i_k$. Using (\ref{eq:symdiff}),
  condition (\ref{eq:vertexflip}) yields 
\begin{equation}\label{eq:outprime}
\outmap_{\Face}(W) \oplus \outmap_{\Face}(W\oplus\{i_t\}) = \{i_1,i_2,\ldots,i_t\}, \quad \forall t\in[k]
\end{equation}
and further
\begin{equation}\label{eq:outprime2}
\outmap_{\Face}(W\oplus\{i_s\} \oplus \outmap_{\Face}(W\oplus\{i_t\}) =
\{i_{s+1},i_{s+2},\ldots,i_t\}, \quad \forall s,t\in[k], s<t.
\end{equation}

In particular, $W$ is complementary to $W\oplus\{i_k\}$ in $\Face$,
but this is the only complementary pair among the $k+1$ 
vertices in $\Face$ that are affected by complementing $W$.
From (\ref{eq:compl}), it similarly follows that
\begin{equation}\label{eq:cyclicshift}
\begin{array}{lclcl}
\outmap'_{\Face}(W) &=& \outmap_{\Face}(W) \oplus \{i_1,i_2,\ldots,i_k\} &\stackrel{(\ref{eq:outprime})}{=}&
  \outmap_{\Face}(W\oplus\{i_k\}), \\
 \outmap'_{\Face}(W\oplus\{i_1\}) &=& \outmap_{\Face}(W\oplus\{i_1\})\oplus\{i_1\}  &\stackrel{(\ref{eq:outprime})}{=}& \outmap_{\Face}(W), \\
\outmap'_{\Face}(W\oplus\{i_t\}) &=& \outmap_{\Face}(W\oplus\{i_t\}) \oplus\{i_{t}\}
                                    &\stackrel{(\ref{eq:outprime2})}{=}& \outmap_{\Face}(W\oplus\{i_{t-1}\}),
                                    
\end{array}
\end{equation}
for $t=2,\ldots,k$.
This means that the $k+1$ affected vertices 
just permute their outmap values under
$\outmap_{\Face}\rightarrow \outmap'_{\Face}$. This does not change
the number of sinks, so $F_{\outmap'}$ has a unique sink as well.

It remains to show that $\Face_{\outmap'}$ is a cap, so
$\Cube_{\outmap'}$ is an odd USO by Lemma~\ref{lem:oddchar}. Since
$\Face_{\outmap}$ is a cap, it suffices to show that complementary
vertices keep odd Hamming distance under
$\outmap_{\Face}\rightarrow \outmap'_{\Face}$. This can also be seen
from (\ref{eq:cyclicshift}): for $t=2,\ldots,k$, the vertex of outmap
value $\outmap_{\Face}(W\oplus\{i_{t-1}\})$ moves by Hamming distance
$2$, namely from $W\oplus\{i_{t-1}\}$ (under $\outmap_{\Face}$) to
$W\oplus\{i_{t}\}$ (under $\outmap'_{\Face}$). Hence it still has odd
Hamming distance to its unaffected complementary vertex. The two
complementary vertices of outmap values $\outmap_{\Face}(W)$ and
$\outmap_{\Face}(W\oplus\{i_k\})$ move by Hamming distance $1$
each. Vertices of other outmap values are unaffected.
\end{proof}

As an example, if we complement the vertex $\overline{Y}$ in the
Klee-Minty cube of Figure~\ref{fig:KM}, we obtain the odd USO in
Figure~\ref{fig:oddface} (left); see
Figure~\ref{fig:complement}. Vertices $\overline{X}$ and
$\overline{Z}$ have moved by Hamming distance 2, while $Y$ and
$\overline{Y}$ have moved by Hamming distance $1$ each. If we
subsequently also complement $W$ (whose neighborhood was unaffected,
so Lemma~\ref{lem:vertexflip} still applies), we obtain another odd
USO (actually, a rotated Klee-Minty cube).

\begin{figure}[htb]
\begin{center}
\includegraphics[width=0.3\textwidth]{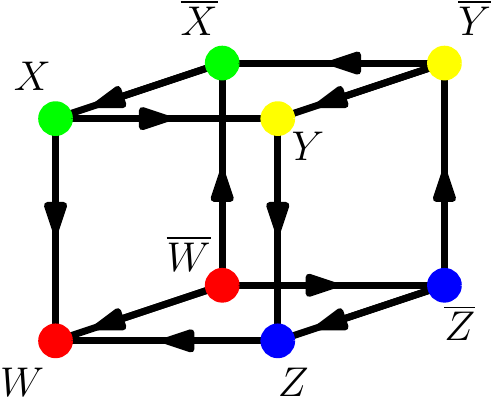}
\includegraphics[width=0.3\textwidth]{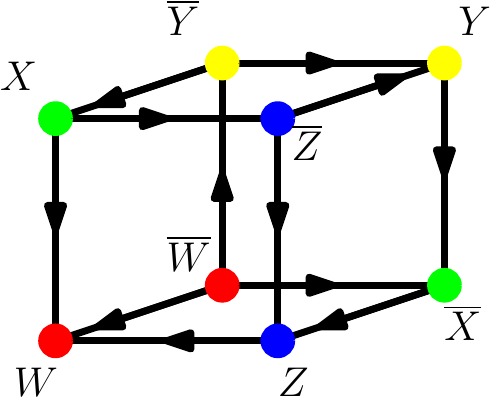}
\includegraphics[width=0.3\textwidth]{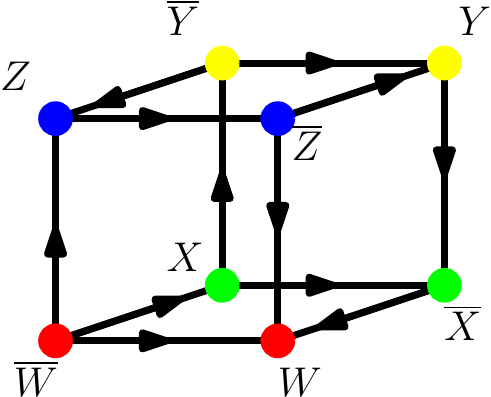}
\end{center}
\caption{Complementing the two vertices $\overline{Y},W$ in
  succession, starting from the Klee-Minty cube (left)}
\label{fig:complement}
\end{figure}

Let $\odd(n)$ denote the number of odd USOs of the standard $n$-cube.
By Definition~\ref{def:odd}, we get 
\begin{equation}\label{eq:oddborder}
\odd(n)=\border(n), \quad \forall n\geq 0,
\end{equation}
as duality (Lemma~\ref{lem:duality}) is a bijection on the set of all USOs.

\section{Counting PUSOs and odd USOs}\label{sec:counting}
With characterizations of USOs, PUSOs, border USOs, and odd USOs
available, one can explicitly enumerate these objects for small
dimensions. Here are the results up to dimension $5$ (the USO column
is due to Schurr~\cite[Chapter 6]{schurr2004diss}). We remark that
most numbers (in particular, the larger ones) have not independently
been verified.

\begin{table}[htb]
    \label{uso_enum_small_n}
    \centering
        \begin{tabular}{|c|r|r|r|}
            \hline
            $n$ & $\uso(n)$ & $\puso(n)=2\border(n-1)$ & $\border(n)=\odd(n)$\\
            \hline
            0 & 1 & 0 & 1\\
            \hline
            1 & 2 & 0 & 2\\
            \hline
            2 & 12 & 4 & 8\\
            \hline
            3 & 744 & 16 & 112\\
            \hline
            4 & 5'541'744 & 224 & 12'928\\
            \hline
            5 & 638'560'878'292'512 & 25'856 & 44'075'264\\            
            \hline
        \end{tabular}
        \caption{The number of USOs, PUSOs and border / odd USOs of the standard
          $n$-cube, obtained through
          computer enumeration}
\end{table}
The number of PUSOs appears to be very small, compared to the total
number of USOs of the same dimension. In this section, we will show
the following asymptotic results that confirms this impression.

\begin{theorem}[Counting PUSOs]\label{thm:pusocount}
Let $\puso(n)$ denote the number of PUSOs of the standard $n$-cube. 
\begin{itemize}
\item[(i)] For $n\geq 2$, $\puso(n) \leq 2^{2^{n-1}}$.
\item[(ii)] For $n\geq 6$, $\puso(n) < 1.777128^{2^{n-1}}$.
\item[(iii)] For $n=2^k, k\geq 2$, $\puso(n) \geq 2^{2^{n-1-\log n}+1}$.
\end{itemize}
\end{theorem}

This shows that the number $\puso(n)$ is doubly exponential but still
negligible compared to the number $\uso(n)$ of USOs of the standard $n$-cube:
Matou\v{s}ek~\cite{matousek2006number} has shown that
\[
\uso(n) \geq \left(\frac{n}{e}\right)^{2^{n-1}},
\]
with a ``matching'' upper bound of $\uso(n) = n^{O(2^n)}$.

As the main technical step, we count odd USOs. We start with the upper bound.

\begin{lemma} \label{lem:bound} Let $n\geq 1$. Then 
\begin{itemize}
\item[(i)] $\odd(n) \leq 2\odd(n-1)^2$ for $n>0$.
\item[(ii)] For $n\geq 2$ and all $k<n$, 
\begin{equation}\label{eq:oddlower}
2\odd(n-1) \leq \left(2\odd(k)\right)^{2^{n-1-k}} =
\sqrt[2^k]{2\odd(k)}^{2^{n-1}}.
\end{equation}
\end{itemize}
\end{lemma}

\begin{proof}
  By Corollary~\ref{cor:odd} (i), every odd USO consists of two odd
  USOs in two opposite facets, and edges along coordinate $n$, say,
  that connect the two facets. We claim that for every choice of odd
  USOs in the two facets, there are at most two ways of connecting the
  facets. Indeed, once we fix the direction of some connecting edge,
  all the others are fixed as well, since the orientation of an edge
  $\{V,V\oplus\{n\}\}$ determines the orientations of all
  ``neighboring'' edges $\{V\oplus\{i\},V\oplus\{i,n\}\}$ via
  Corollary~\ref{cor:odd} (ii) (all $2$-faces are bows). Inequality
  (i) follows, and (ii) is a simple induction. 
\end{proof}

The three bounds on $\puso(n)$ now follow
from $\puso(n)=2\border(n-1)$ (\ref{eq:pusoborder}) and
$\border(n-1)=\odd(n-1)$ (\ref{eq:oddborder}). For the bound of
Theorem~\ref{thm:pusocount} (i), we use (\ref{eq:oddlower}) with
$k=0$, and for Theorem~\ref{thm:pusocount} (ii), we employ $k=5$ and
$\odd(5)= 44'075'264$. The lower bound of Theorem~\ref{thm:pusocount} (iii)
is a direct consequence of the following ``matching'' lower bound on
the number of odd USOs.

\begin{lemma}\label{lem:odd_lower}  Let $n=2^k, k\geq 1$. Then
$\odd(n-1) \geq 2^{2^{n-1-\log n}}$.
\end{lemma}

\begin{proof} If $n=2^k$, there exists a perfect Hamming code of block
  length $n-1$ and message length $n-1-\log n$~\cite{Hamming}. In our
  language, this is a set ${\cal W}$ of $2^{n-1-\log n}$ vertices of
  the standard $(n-1)$-cube, with pairwise Hamming distance $3$ and
  therefore disjoint neighborhoods. Hence, starting from the
  Klee-Minty cube $\KM^{[n-1]}$ as introduced in Lemma~\ref{lem:km},
  we can apply Lemma~\ref{lem:vertexflip} to get a different odd USO
  for every subset of ${\cal W}$, by complementing all
  vertices in the given subset. The statement follows.
\end{proof} 

\section{Conclusion}\label{sec:conclusion}
In this paper, we have introduced, characterized, and (approximately)
counted three new classes of $n$-cube orientations: pseudo unique sink
orientations (PUSOs), border unique sink orientations (facets of
PUSOs), and odd unique sink orientations (duals of border USOs). A
PUSO is a dimension-minimal witness for the fact that a given cube
orientation is not a USO. The requirement of minimal dimension induces
rich structural properties and a PUSO frequency that is negligible
compared to the frequency of USOs among all cube orientations.

An obvious open problem is to close the gap in our approximate
counting results and determine the true asymptotics of
$\log\odd(n)$ and hence $\log\puso(n)$. We have shown that these
numbers are between $\Omega(2^{n-\log n})$ and $O(2^n)$. As our lower
bound construction based on the Klee-Minty cube seems to yield rather
specific odd USOs, we believe that the lower bound can be improved.

Also, border USOs and odd USOs might be algorithmically more tractable
than general USOs. The standard complexity measure here is the number
of outmap values\footnote{provided by an oracle that can be invoked
  for every vertex} that need to be inspected in order to be able to
deduce the location of the sink~\cite{szabo2001unique}. For example,
in dimension $3$, we can indeed argue that border USOs and odd USOs
are easier to solve than general USOs. It is known that 4 outmap
values are necessary and sufficient to locate the sink in any USO of
the 3-cube~\cite{szabo2001unique}. But in border USOs and odd USOs of
the 3-cube, 3 suitably chosen outmap values suffice to deduce the
orientations of all edges and hence the location of the
sink~\cite{wff}; see Figure~\ref{fig:4steps}. 

\begin{figure}[htb]
\begin{center}
\includegraphics[width=0.65\textwidth]{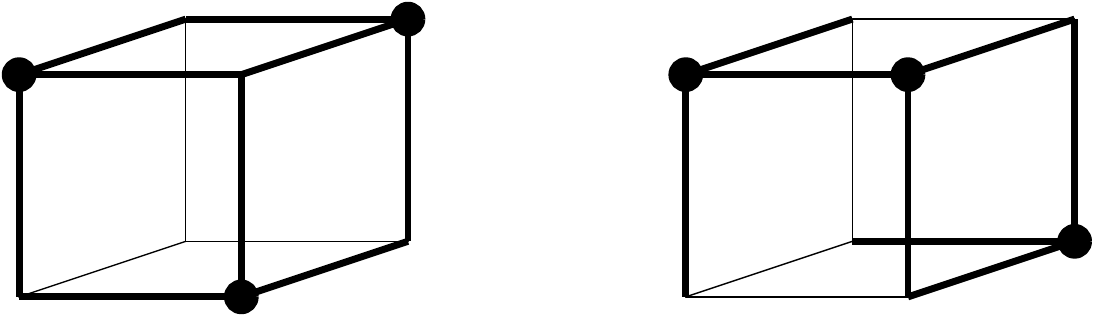}
\end{center}
\caption{The outmap values of the 3 indicated vertices determine the
  orientations of the bold edges. In the case of a border USO (left),
  the remaing orientations are determined by the condition that
  antipodal vertices have different outmap parities
  (Theorem~\ref{thm:pusof_charact}). In the case of an odd USO
  (right), the remaining orientations are determined by the condition that
  all 2-faces are bows (Corollary~\ref{cor:odd}(ii)).}
\label{fig:4steps}
\end{figure}

\bibliography{puso_arxiv}
\bibliographystyle{plain}

\end{document}